\documentclass{amsart}
\usepackage[utf8]{inputenc}
\usepackage[T1]{fontenc}
\usepackage{cite}
\usepackage{amsfonts}
\usepackage{amssymb}
\usepackage{amsthm}
\usepackage{enumitem}
\usepackage{amsmath}
\usepackage{graphicx}
\usepackage{stmaryrd}
\usepackage{csquotes}
\usepackage{aliascnt}
\usepackage{mathtools}
\usepackage{bbm}
\usepackage[unicode,bookmarksopen]{hyperref}

\newcommand{\autorefre}[1]{\hyperref[re:#1]{\autoref*{#1}}}

\newlist{multistmt}{enumerate}{1}
\setlist[multistmt]{label={\upshape(\arabic*)}, nosep}

\newlist{propertylist}{enumerate}{1}
\setlist[propertylist]{label={\upshape(\roman*)}, nosep}

\newtheorem{theorem}{Theorem}[section]

\newcommand{\MakeTheoremAndCounter}[2]{\newaliascnt{#1}{theorem}
\newtheorem{#1}[#1]{#2}
\aliascntresetthe{#1}
\expandafter\providecommand\csname#1autorefname\endcsname{#2}}

\MakeTheoremAndCounter{lemma}{Lemma}
\MakeTheoremAndCounter{claim}{Claim}
\MakeTheoremAndCounter{example}{Example}
\MakeTheoremAndCounter{fact}{Fact}
\MakeTheoremAndCounter{question}{Question}
\MakeTheoremAndCounter{corollary}{Corollary}
\MakeTheoremAndCounter{notation}{Notation}
\MakeTheoremAndCounter{problem}{Problem}
\MakeTheoremAndCounter{proposition}{Proposition}
\MakeTheoremAndCounter{conjecture}{Conjecture}
\newtheorem*{restatemain}{\autoref{mainthm}}

\theoremstyle{definition}
\MakeTheoremAndCounter{definition}{Definition}

\theoremstyle{remark}
\newtheorem*{remark}{Remark}

\newcommand{\modulo}[1]{~(\mathrm{mod}~#1)}
\DeclareMathOperator{\restrict}{\upharpoonright}

\DeclareMathOperator{\concat}{{}^{\smallfrown}}

\newcommand{\titleofthepaper}{A Haar meager set that is not strongly Haar meager}
\title{\titleofthepaper}

\author{M\'arton Elekes}
\address{Alfr\'ed R\'enyi Institute of Mathematics, Hungarian Academy of Sciences,
PO Box 127, 1364 Budapest, Hungary and E\"otv\"os Lor\'and
University, Institute of Mathematics, P\'azm\'any P\'eter s. 1/c,
1117 Budapest, Hungary}
\email{elekes.marton@renyi.mta.hu}
\urladdr{http://www.renyi.hu/$\sim$emarci}

\author{Don\'at Nagy}
\address{E\"otv\"os Lor\'and University, Institute of Mathematics, P\'azm\'any P\'eter s. 1/c, 1117 Budapest, Hungary}
\email{nagdon@bolyai.elte.hu}

\author{M\'ark Po\'or}
\address{E\"otv\"os Lor\'and University, Institute of Mathematics, P\'azm\'any P\'eter s. 1/c, 1117 Budapest, Hungary}
\email{sokmark@gmail.com}

\author{Zolt\'an Vidny\'anszky}
\address{Kurt G\"odel Research Center for Mathematical Logic, Universit\"at Wien, W\"ah\-rin\-ger Strasse 25, 1090 Wien, Austria}
\email{zoltan.vidnyanszky@univie.ac.at}

\thanks{All four authors were supported by the National Research, Development and Innovation Office -- NKFIH, grants no.~104178 and 124749. The first and fourth authors were also supported by the National Research, Development and Innovation Office -- NKFIH, grant no.~113047. The fourth author was also supported by FWF Grant P29999.\\
\includegraphics[height=1cm]{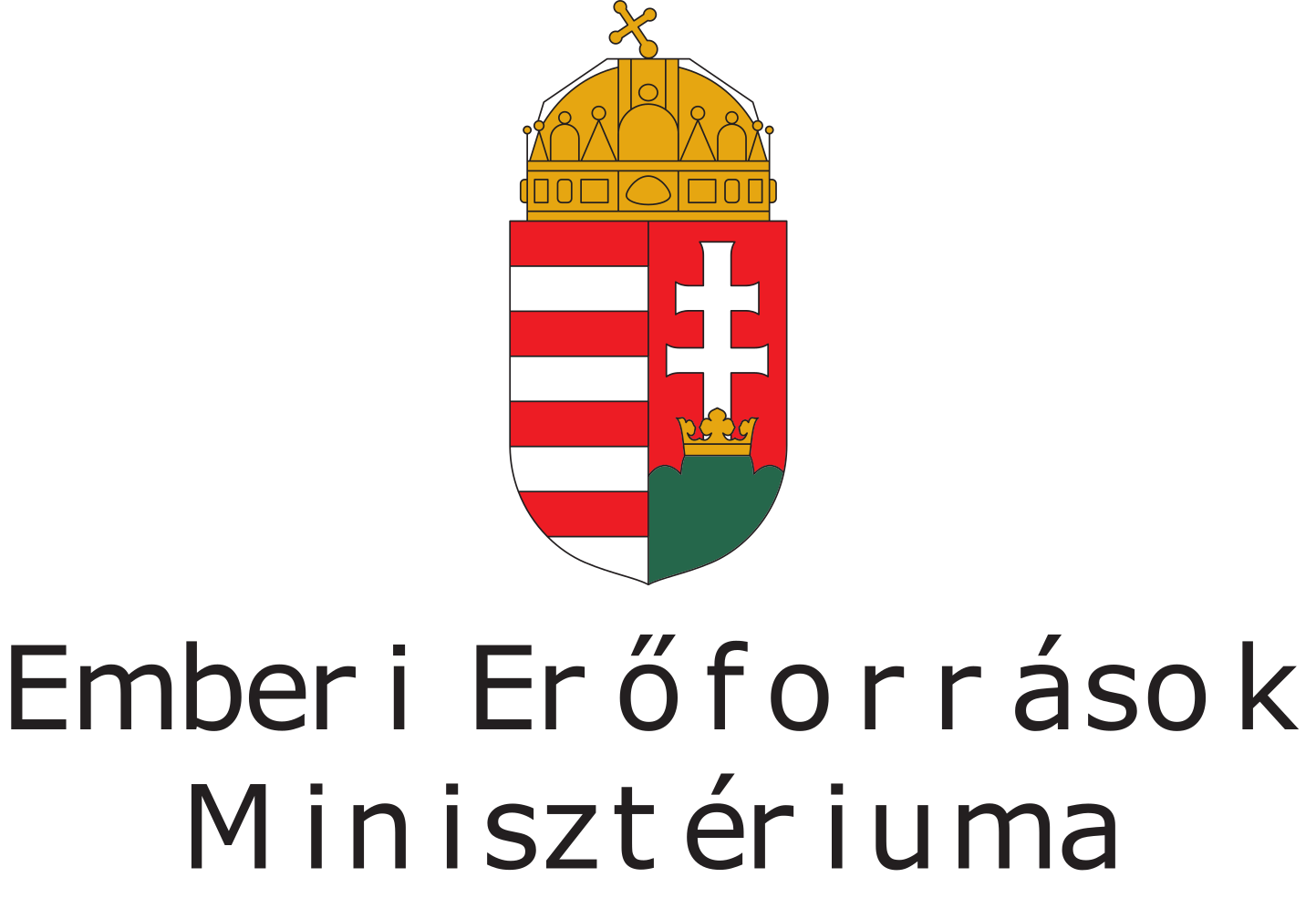} \raisebox{0.5cm}{\parbox[c]{11cm}{\scshape Supported by the \'UNKP-17-3 New National Excellence Program of the Ministry of Human Capacities.}}}
%, 28C10, 22F99
\begin{document}
\begin{abstract}
Following Darji, we say that a Borel subset $B$ of an abelian Polish group $G$ is \emph{Haar meager} if there is a compact metric space $K$ and a continuous function $f : K \to G$ such that the preimage of the translate, $f^{-1}(B+g)$ is meager in $K$ for every $g \in G$. The set $B$ is called \emph{strongly Haar meager} if there is a compact set $C \subseteq G$ such that $(B+g) \cap C$ is meager in $C$ for every $g \in G$. The main open problem in this area is Darji's question asking whether these two notions are the same. Even though there have been several partial results suggesting a positive answer, in this paper we construct a counterexample. More specifically, we construct a $G_\delta$ set in $\mathbb{Z}^\omega$ that is Haar meager but not strongly Haar meager. We also show that no $F_\sigma$ counterexample exists, hence our result is optimal.
\end{abstract}

\maketitle

\tableofcontents

\section{Introduction}

The notion of Haar meager sets in Polish groups was first introduced by Darji in \cite{Da} in 2013 as a topological counterpart to the so-called Haar null sets. Although in this paper we will only study abelian Polish groups, we note that \cite{DRVV} have generalized this notion for arbitrary Polish groups. For a recent survey about the basic properties of Haar null sets and Haar meager sets, see \cite{EN}.

\begin{definition}\label{def:haarmeager}
Let $(G, +)$ be an abelian Polish group. A set $A \subseteq G$ is said to be \emph{Haar meager} if there are a Borel set $B \supseteq A$, a nonempty compact metric space $K$ and a continuous function $f : K \to G$ such that $f^{-1}(B+g)$ is meager in $K$ for every $g \in G$. A function $f$ satisfying this is called a \emph{witness function} for $A$.
\end{definition}

Haar meager sets form a $\sigma$-ideal which is contained in the $\sigma$-ideal of meager sets (see \cite[Theorems 2.9 and 2.2]{Da}). In a locally compact group, these two ideals coincide, but if $G$ is not locally compact, then there exists a closed meager set that is not Haar meager (see \cite[Corollary 2.5 and Example 2.6]{Da}).

When we prove that a set is Haar meager, it is often possible to use a witness function which is just the identity of $G$ restricted to some compact subset of $G$. This observation yields the following:

\begin{definition}\label{def:stronglyhaarmeager}
Let $(G, +)$ be an abelian Polish group. A set $A \subseteq G$ is said to be \emph{strongly Haar meager} if there are a Borel set $B \supseteq A$ and a nonempty compact set $C\subseteq G$ such that $(B+g)\cap C$ is meager in $C$ for every $g \in G$. A compact set $C$ satisfying this is called a \emph{witness set} for $A$.
\end{definition}

\begin{remark}
It is easy to see that a set $A\subseteq G$ is strongly Haar meager if and only if it is Haar meager and has an injective witness function.
\end{remark}

One of the most important questions in the topic is \cite[Problem 2]{Da}, which asks if these two notions are equivalent:

\begin{question}[Darji]\label{que:HMeqSHM}
Is every Haar meager set strongly Haar meager?
\end{question}

We remark that in the case of Haar null sets \cite[Theorem 4.3]{BGJS} answers the analogous question affirmatively.

To answer \autoref{que:HMeqSHM} affirmatively, it would have been a natural idea to prove that if $f: K\to G$ is a witness function for a Haar meager set, then the choice $C=f(K)$ satisfies the requirements of \autoref{def:stronglyhaarmeager}. However \cite[Example 11]{DRVV} ruled out this by constructing a Haar meager set where the image of one particular witness function does not satisfy the requirements of \autoref{def:stronglyhaarmeager}:

\begin{example}[Dole\v zal-Rmoutil-Vejnar-Vlas\' ak]\label{exa:shmproofproblem}
There exists a $G_\delta$ Haar meager set $A\subseteq \mathbb{R}$, a compact metric space $K$ and a witness function $f:K\to \mathbb{R}$ such that $A\cap f(K)$ is comeager in $f(K)$.
\end{example}

The result \cite[Theorem 5.13]{BGJS} shows that in a certain class of Polish groups the Haar meager sets and the strongly Haar meager sets coincide:

\begin{definition}
A topological group $G$ is called \emph{hull-compact} if each compact subset of $G$ is contained in a compact subgroup of $G$.
\end{definition}

\begin{theorem}[Banakh-G\l\k{a}b-Jab\l o\'nska-Swaczyna]\label{bgjsthm}
If the abelian Polish group $G$ is hull-compact, then every Haar meager subset of $G$ is strongly Haar meager.
\end{theorem}

\begin{example}\label{bgjsexample}
It is not very hard to verify that the abelian Polish group $(\mathbb{Q}/\mathbb{Z})^\omega$ is hull-compact (where we endow $\mathbb{Q}$ with the discrete topology).
\end{example}

In this paper we answer the question of Darji negatively by constructing a counterexample in the group $\mathbb{Z}^\omega$:

\begin{theorem}\label{mainthm}
In the abelian Polish group $\mathbb{Z}^\omega$, there exists a $G_\delta$ set $R$ that is Haar meager but not strongly Haar meager.
\end{theorem}

In \autoref{mincompl} we also prove that this counterexample is as simple as possible: every $F_\sigma$ Haar meager set is strongly Haar meager. We also note that in \autoref{rhm} we prove that our example is in fact a so-called Haar nowhere dense set.

As an additional motivation, we mention that according to part (2) of \cite[Theorem 13.8]{BGJS} the \emph{generic} variants of these notions (when we require not only one witness, but comeager many witnesses) do coincide:
\begin{theorem}[Banakh-G\l\k{a}b-Jab\l o\'nska-Swaczyna]
Assume that $G$ is an abelian Polish group and $B\subseteq G$ is a Borel set. Then the following are equivalent:
\begin{multistmt}
\item in the Polish space $\mathcal{C}(\{0,1\}^\omega, G)$ of continuous functions from $\{0,1\}^\omega$ to $G$ (endowed with the compact-open topology) the set
\[\{f\in \mathcal{C}(\{0,1\}^\omega, G) : \text{$f$ witnesses that $B$ is Haar meager}\}\quad\text{is comeager},\]
\item in the Polish space $\mathcal{K}(G)$ of nonempty, compact subsets of $G$
\[\{C\in \mathcal{K}(G) : \text{$C$ witnesses that $B$ is strongly Haar meager}\}\quad\text{is comeager}.\]
\end{multistmt}
\end{theorem}
In (1) we only considered the potential witness functions whose domain is the Cantor set $\{0,1\}^\omega$. This is a natural restriction, as according to \cite[Proposition 3]{DV}, every Haar meager set has a witness function whose domain is $\{0,1\}^\omega$.

\section{Preliminaries}

As usual, $\mathbb{N}$ and $\omega$ will both denote the set of nonnegative integers. We will use \enquote{$\mathbb{N}$} when we use this set as a topological space (with the discrete topology) and use \enquote{$\omega$} when we use it as an ordinal or index set.

We will use some notation related to sequences (i.e.\ functions $s$ whose domain is either $\omega$ or $\{0,1,\ldots, n-1\}$ for some natural number $n$). As usual, $s_k$ denotes the element of the sequence $s$ with index $k$ (i.e. $s_k=s(k)$ for an index $k$ that is in the domain of $s$). If $S$ is an arbitrary set, then $S^{<\omega}=\bigcup_{n\in\omega}S^n$ denotes the set of finite sequences of elements of $S$ and $\emptyset$ denotes the empty sequence $\emptyset\in S^0$. For $s\in S^{<\omega}$, $|s|$ denotes the length of $s$.

If $s$ and $s'$ are two sequences and $s$ is finite ($s'$ may be infinite), then $s\concat s'$ denotes concatenation of $s$ and $s'$. In particular if $s\in S^{<\omega}$ and $\ell \in S$ is an additional element, then $s\concat \ell$ denotes the sequence of length $|s|+1$ which consists of the elements of $s$ followed by $\ell$ as the last element. If $x$ is a (finite or infinite) sequence of length at least $n$, then $x\restrict n$ denotes the sequence formed by the first $n$ elements of $x$.

If $s, s'$ are two sequences then we say that $s\subseteq s'$ if there is a sequence $t$ such that $s' = s \concat t$ (if we consider functions and sequences as sets of pairs, this is just the usual inclusion relation). For a sequence $s\in S^{<\omega}$, let $[s]\subseteq S^{\omega}$ be the set of sequences which have $s$ as an initial segment, i.e.
\[[s]=\{x\in S^\omega : s\subseteq x\}.\]

We say that a set $A\subseteq S^{<\omega}$ is cofinal if for every $s\in S^{<\omega}$ there exists an $a\in A$ such that $s\subseteq a$ (i.e. $s$ is an initial segment of $a$).

We will use $\mathcal{K}(X)$ to denote the \emph{nonempty} compact sets of a space $X$, equipped with the Vietoris topology. The well-known result \cite[Theorem 4.25]{Ke} states that if $X$ is Polish, then $\mathcal{K}(X)$ is also Polish. We will also use the following fact:

\newcommand{\michaelref}{Michael \cite[4.13.1]{Mi}}
\begin{fact}[\michaelref]\label{kxzerodim}
If $X$ is a zero-dimensional Polish space (that is, $X$ is Polish and has a basis consisting of clopen sets), then $\mathcal{K}(X)$ is also zero-dimensional.
\end{fact}

%\begin{proof}
%Notice that we may fix a basis $(B_i)_{i\in I}$ of $X$ such that $B_i$ is clopen for each $i\in I$ and it forms an algebra (i.e. it is closed under complements and finite unions). By definition, the Vietoris topology is generated by the subbase consisting of the sets
%\[U^+ = \{K\in \mathcal{K}(X) : K\cap U \neq \emptyset\}\quad\text{and}\quad V^- = \{K\in\mathcal{K}(X) : K \subseteq V\}\]
%where $U$ and $V$ range over the open subsets of $X$. However, it is not very hard to verify that the subbase consisting of the sets
%\[B_i^+ = \{K\in \mathcal{K}(X) : K\cap B_i \neq \emptyset\}\quad\text{and}\quad B_j^- = \{K\in\mathcal{K}(X) : K \subseteq B_j\}\]
%where $i, j\in I$ also generates the Vietoris topology. These sets are clearly clopen, and so their finite intersections are also clopen. This shows that $\mathcal{K}(X)$ is also zero-dimensional.
%\end{proof}

A standard reference book for notions in descriptive set theory is \cite{Ke}.

\section{Translating the compact sets apart}

This section is motivated by the results and ideas in \cite{Sfst,Ssnd,EV}. These papers prove slightly weaker claims than our \autoref{translateapart}, but work in a more general setting. (Our proof relies on the structure of $\mathbb{Z}^\omega$ to make the calculations shorter and simpler.)

Let \[H=\{(K,x) \in \mathcal{K}(\mathbb{Z}^\omega) \times \mathbb{Z}^\omega: x \in K\}.\]

Note that $H$ is a closed set in the product space $\mathcal{K}(\mathbb{Z}^\omega) \times \mathbb{Z}^\omega$.

\begin{theorem}\label{translateapart}
There exists a map $t: \mathcal{K}(\mathbb{Z}^\omega) \to  \mathbb{Z}^\omega$ so that the map $T: H \to \mathbb{Z}^\omega$ defined by \[T(K,x)=x+t(K)\] is a homeomorphism between $H$ and the set 
\[F=T(H) = \bigcup_{K\in\mathcal{K}(\mathbb{Z}^\omega)} (K+t(K)),\]
where the union is disjoint. Moreover, $F$ is a closed subset of $\mathbb{Z}^\omega$ and satisfies that
\[(K+t(K)+\{-1,0,1\}^\omega)\cap F = K+t(K)\]
for each $K\in\mathcal{K}(\mathbb{Z}^\omega)$.
\end{theorem}

\begin{proof}
As \autoref{kxzerodim} states that $\mathcal{K}(\mathbb{Z}^\omega)$ is zero-dimensional, we may apply \cite[Theorem 7.2]{Ke} to get an embedding
\[c: \mathcal{K}(\mathbb{Z}^\omega) \to \{-1,1\}^\omega.\]
Define the (clearly continuous) function
\[b:\mathcal{K}(\mathbb{Z}^\omega) \to \omega^\omega,\qquad b(K)_n=\max\{\lvert x_n\rvert : x\in K\}+1.\]

For each $n\in \omega$ let
\[t(K)_n=3\cdot b(K)_n\cdot c(K)_n.\]
It is clear that both $t$ and
\[T : H \to \mathbb{Z}^\omega,\qquad T(K, x) = x +t(K)\]
are continuous.

\begin{fact}\label{ttriv}
Assume that $(K, x)\in H$, $y = T(K, x)$ and $n\in\omega$. Then it is easy to verify that
\begin{multistmt}
\item $\lvert y_n\rvert > b(K)_n$ and
\item \(\displaystyle c(K)_n = \begin{cases} +1 & \text{if $y_n>0$,}\\ -1 & \text{if $y_n<0$.}\end{cases}\)
\end{multistmt}
\end{fact}

\begin{claim}
$T$ is injective.
\end{claim}
\begin{proof}
Assume that $(K,x), (K',x')\in H$ and $T(K,x)=T(K',x')$. Using \autoref{ttriv} part (2), $c(K)=c(K')$, but then $K=K'$, as $c$ is injective. This also implies that
\[x = T(K, x) - t(K) = T(K', x')- t(K')=x'.\qedhere\]
\end{proof}

\begin{claim}
$F=T(H)$ is closed and the map $T^{-1} : F \to H$ is continuous.
\end{claim}
\begin{proof}
Assume that $y^{(m)}\in F$ for each $m\in\omega$ and this sequence converges to some $y^*\in\mathbb{Z}^\omega$. As $F=T(H)$, there are compact sets $K^{(m)}\in\mathcal{K}(\mathbb{Z}^\omega)$ and sequences $x^{(m)}\in K^{(m)}$ such that $y^{(m)} = T(K^{(m)}, x^{(m)})$ for each $m\in\omega$. It is sufficient to prove that $\big((K^{(m)}, x^{(m)})\big)_{m\in\omega}$ converges to some $(K^*, x^*)\in H$. (If this holds, then $y^*=T(K^*, x^*)\in F$ also demonstrates that $F$ is closed.)

Using \autoref{ttriv} part (2) and the convergence of $y^{(m)}$ yields that $\big(c(K^{(m)})\big)_{m\in\omega}$ converges to some element $\gamma\in\{-1,1\}^\omega$. Our next step is to prove that $\gamma$ is contained in the image of $c$.

Notice that 
\[f: \omega\to\omega,\qquad f(n) = \sup_{m\in\omega} \lvert y^{(m)}_n\rvert\]
is a well-defined function, because for each $n\in\omega$ the sequence $\big(\lvert y^{(m)}_n\rvert\big)_{m\in\omega}$ is convergent and therefore bounded. Using \autoref{ttriv} part (1),
\[f(n)\ge \lvert y^{(m)}_n\rvert> b(K^{(m)})_n\qquad\text{for each $n, m\in\omega$.}\]
Applying the definition of $b$, this implies that for each $m\in\omega$,
\[K^{(m)}\in\mathcal{K}(\underbrace{\{z \in\mathbb{Z}^\omega : \lvert z_n\rvert < f(n)\text{ for each $n\in\omega$}\}}_\text{compact}).\]
As the (nonempty) compact subsets of a compact space form a compact space themselves, $\big(K^{(m)}\big)_{m\in\omega}$ has a subsequence that converges to some compact set $K^*$. Applying this, the continuity of $c$ and the fact that $\lim_{m\in\omega} c(K^{(m)}) = \gamma$ exists, we obtain that
\[\gamma = \lim_{m\in\omega} c(K^{(m)})= c(K^*).\]
As $c$ was an embedding, this implies that $K^{(m)}$ converges to $K^*$ (when $m\to\infty$).

As $t$ is continuous and $y^{(m)}$ is convergent, this implies that $x^{(m)} = y^{(m)}- t(K^{(m)})$ is also convergent to some $x^* \in \mathbb{Z}^\omega$. Finally $(K^*, x^*)\in H$ follows from the fact that $(K^{(m)}, x^{(m)})\to (K^*, x^*)$ and $H$ is a closed set.
\end{proof}

This claim implies that $F$ is closed and $T$ is a homeomorphism between $H$ and $F$. It is clear that
\[F = \bigcup_{K\in\mathcal{K}(\mathbb{Z}^\omega)} (K+t(K)).\]

To conclude the proof, we will fix an arbitrary $K\in\mathcal{K}(\mathbb{Z}^\omega)$ and show that
\[(K+t(K)+\{-1,0,1\}^\omega)\cap F = (K+t(K)).\]
Fix an arbitrary $z\in (K+t(K)+\{-1,0,1\}^\omega)\cap F$. There are $x\in K$, $\varepsilon \in \{-1,0,1\}^\omega$, $K'\in\mathcal{K}(\mathbb{Z}^\omega)$ and $x'\in K'$ such that
\[z = \underbrace{x + t(K)}_{=H(K,x)} + \varepsilon = \underbrace{x' + t(K')}_{=H(K', x')}.\]
Using both parts of \autoref{ttriv} and the fact that $b(K)_n\ge 1$ for each $n\in\omega$ it is easy to verify that
\[c(K)_n = 1 \quad\Leftrightarrow\quad z_n >0 \quad\Leftrightarrow\quad c(K')_n = 1\quad\text{for each $n\in\omega$}.\]
As $c$ is injective, this means that $K' = K$, so $z \in K+t(K)$, as claimed.
\end{proof}

\begin{remark}
An analogous proof would work in the case when $\{-1,0,1\}^\omega$ is replaced in the statement by any other compact subset $C\subseteq\mathbb{Z}^\omega$ that contains the all zero sequence.
\end{remark}

\section{Construction of the witness function}
The goal of this section is to construct a compact metric space $K_0$ and a function $f_0 : K_0\to \mathbb{Z}^\omega$ that will witness the Haar meagerness of our example. In order to do this, we will introduce some structures and prove elementary claims about their properties.

We say that a sequence $s$ is \emph{$m$-segmented} if it is the concatenation of constant sequences of length $2^m$. More formally, this means the following:

\begin{definition}
Let $S$ be an arbitrary set and $m\in\mathbb{N}$ be a nonnegative integer. An infinite sequence $s\in S^\omega$ is $m$-segmented if $s_{q \cdot 2^m + r_1} = s_{q\cdot 2^m + r_2}$ for all integers $q\in \mathbb{N}$ and $0\le r_1, r_2 < 2^m$. A finite sequence $s\in S^{<\omega}$ is $m$-segmented if $\lvert s \rvert = Q\cdot 2^m$ for some $Q\in\mathbb{N}$ and $s_{q \cdot 2^m + r_1} = s_{q\cdot 2^m + r_2}$ for all integers $0\le q <Q$ and $0\le r_1, r_2 < 2^m$.
\end{definition}

Notice that an $m$-segmented sequence is also $m'$-segmented for all $0\le m'\le m$.

\begin{definition}\label{bsdef}
Let us fix a sequence $b_s\in \{0, 1\}^{<\omega}$ for each $s\in\omega^{<\omega}$ in a way that it satisfies the following properties:
\begin{multistmt}
\item $b_\emptyset = \emptyset$,
\item $\lvert b_s\rvert$ is divisible by $2^{\lvert s \rvert}$,
\item if $s\in\omega^{<\omega}$ and $\beta$ is a nonempty finite $\lvert s\rvert$-segmented sequence such that $\lvert b_s\rvert + \lvert \beta\rvert$ is divisible by $2^{\lvert s\rvert +1}$, then there is exactly one $\ell\in\omega$ such that $b_{s\concat \ell} = b_s\concat \beta$,
\item conversely, if $s\in\omega^{<\omega}$ and $\ell\in\omega$, then $b_{s\concat\ell}$ can be written as $b_{s\concat \ell} = b_s\concat \beta$ where $\beta$ is a nonempty finite $\lvert s\rvert$-segmented sequence satisfying that $\lvert b_s\rvert + \lvert \beta\rvert$ is divisible by $2^{\lvert s\rvert +1}$.
\end{multistmt}
\end{definition}
It is clear that using recursion we can choose a system $\{b_s\}_{s\in\omega^{<\omega}}$ that satisfies these properties.

\begin{definition}\label{csdef}
Define the set $C_s\subseteq \{0,1\}^\omega\subseteq \mathbb{Z}^\omega$ by
\[C_s = \{ b_s\concat x : x\in\{0,1\}^\omega\text{ and $x$ is $\lvert s\rvert$-segmented}\}.\]
\end{definition}

\begin{fact}\label{cstriv}
The sets $C_s$ have the following properties:
\begin{multistmt}
\item $C_\emptyset= \{0,1\}^\omega$,
\item $C_s$ is compact for each $s\in\omega^{<\omega}$,
\item if $s\subseteq s'$, then $C_s\supseteq C_{s'}$.
\end{multistmt}
\end{fact}
\begin{proof} Properties (1) and (2) are trivial, property (3) can be proved by using induction on the value of $(\lvert s'\rvert -\lvert s\rvert)$ and applying property (4) of \autoref{bsdef}.
\end{proof}

\begin{claim}\label{csprop}
Assume that $s\in\omega^{<\omega}$ and $U$ is a nonempty relatively open subset of $C_s$. Then there are infinitely many indices $\ell\in\omega$ such that $C_{s\concat\ell}\subseteq U$.
\end{claim}
\begin{proof}
Fix an arbitrary element $u\in U$. As $U$ is relatively open, there exists an $n_0\in\omega$ such that $C_s\cap [u\restrict n]\subseteq U$ for each $n\ge n_0$. We may also assume that $n_0> \lvert b_s\rvert$.

Let us define the infinite set
\[N= \{n\in\omega : n \ge n_0\text{ and } 2^{\lvert s\rvert +1}\text{ divides }n\}\] 
and consider an arbitrary $n\in N$. 

As $U\subseteq C_s$, the sequence $u$ can be written as $b_s\concat x$ where $x\in\{0,1\}^\omega$ is an $\lvert s\rvert$-segmented sequence. Using property (2) of \autoref{bsdef}, $2^{\lvert s\rvert}$ divides $n-\lvert b_s\rvert$, therefore it is clear that $\beta = x \restrict (n-\lvert b_s\rvert)$ is $\lvert s\rvert$-segmented. Also notice that $\beta$ is nonempty because $n\ge n_0>\lvert b_s \rvert$ and $\lvert b_s\rvert +\lvert \beta \rvert = n$ is divisible by $2^{\lvert s\rvert +1}$, and thus by property (3) of \autoref{bsdef}, there exists an index $\ell_n\in\omega$ such that
\[u\restrict n = b_s\concat (x\restrict (n-\lvert b_s\rvert)) = b_s\concat\beta = b_{s\concat\ell_n}.\]

Clearly $C_{s\concat\ell_n}\subseteq [b_{s\concat\ell_n}]$ and according to property (3) of \autoref{cstriv} $C_{s\concat\ell_n}\subseteq C_s$, therefore $C_{s\concat\ell_n}\subseteq U$. This is sufficient, because property (3) of \autoref{bsdef} implies that $\ell_n\neq \ell_{n'}$ if $n\neq n'$.
\end{proof}

We will also use the following encoding of $\omega^{<\omega}$ in $\{0,1\}^{<\omega}$:
\begin{definition}\label{hsdef}
For $n\in\omega$, let $h(n) = (0, 0, \ldots, 0, 1)\in \{0,1\}^{n+1}$ be the sequence consisting of $n$ zeroes and then \enquote{1} as the last element. Generalizing this, for a finite sequence $s=(s_0, s_1, \ldots, s_{k-1})\in\omega^{<\omega}$, let $h(s)=h(s_0)\concat h(s_1)\concat \ldots \concat h(s_{k-1})$.

Let $h_0(s)\in\{0,1\}^\omega$ denote the infinite sequence $h(s)\concat (0,0, \ldots)$ (that is, $h(s)$ extended to infinite length by appending zeroes).
\end{definition}

\begin{fact}\label{hsprop}
The functions $h: \omega^{<\omega}\to \{0,1\}^{<\omega}$ and $h_0: \omega^{<\omega}\to \{0,1\}^\omega$ have the following properties:
\begin{multistmt}
\item $h$ and $h_0$ are both injective,
\item if $s, s'\in\omega^{<\omega}$, then
\[s\subseteq s' \quad\Leftrightarrow\quad [h(s)] \supseteq [h(s')] \quad\Leftrightarrow\quad [h(s)]\owns h_0(s'),\]
\item if $s\in\omega^{<\omega}$, $k\in\omega$ and $\ell\in\omega$ is large enough, then $[h(s\concat \ell)] \subseteq [h_0(s)\restrict k]$.
\end{multistmt}
\end{fact}

Intuitively we will use $h_0(s)$ as a \enquote{height} assigned to the sequence $s$: when we construct the compact metric space $K_0$, we will \enquote{lift} a copy of $C_s$ to height $h_0(s)$ for each $s\in\omega$. More precisely, this means the following:

\begin{definition}\label{kzfzdef}
Let $K_0$ be the closure of the set
\[\bigcup_{s\in\omega^{<\omega}} C_s \times \{h_0(s)\}\]
in the space $\mathbb{Z}^\omega \times \{0,1\}^\omega$ and let $f_0 : K_0\to\mathbb{Z}^\omega$ be the restriction of the projection $\mathbb{Z}^\omega \times \{0,1\}^\omega \to \mathbb{Z}^\omega$ to the set $K_0$.
\end{definition}

\begin{claim}$K_0$ is compact.
\end{claim}
\begin{proof}
According to \autoref{csdef}, $C_s\subseteq \{0,1\}^\omega\subseteq \mathbb{Z}^\omega$ for each $s\in\omega^{<\omega}$. This clearly implies that the closed set $K_0$ is a subset of $\{0,1\}^\omega\times\{0,1\}^\omega$, a compact set.
\end{proof}

To study the structure of the set $K_0$, we introduce
\[D_s = \bigcup_{\substack{\sigma\in\omega^{<\omega}\\\sigma\supseteq s}} C_\sigma \times \{h_0(\sigma)\} \quad \subseteq \mathbb{Z}^\omega \times \{0,1\}^\omega.\] 

Using this notation, the definition of $K_0$ can be written as $K_0=\overline{D_\emptyset}$. It is clear that if $s\supseteq s'$ then $D_s\subseteq D_{s'}$, and in particular $D_s\subseteq D_\emptyset \subseteq K_0$ holds for each $s\in\omega^{<\omega}$.

\begin{claim}\label{oldsclopen}
For each $s\in\omega^{<\omega}$,
\[\overline{D_s} = K_0\cap (\mathbb{Z}^\omega \times [h(s)]) = K_0\cap (C_s \times [h(s)])\]
and therefore $\overline{D_s}$ is relatively clopen in $K_0$.
\end{claim}

\begin{proof}
By property (2) of \autoref{hsprop}, $s\subseteq \sigma$ if and only if $h_0(\sigma)\in [h(s)]$, so
\[D_\emptyset \cap (\mathbb{Z}^\omega \times [h(s)]) = D_s.\]
Elementary calculations show that if $A, B$ are two subsets of a topological space and $B$ is clopen, then $\overline{A\cap B} = \overline{A} \cap B$. Using this for the clopen set $\mathbb{Z}^\omega \times [h(s)]$ yields that
\[K_0 \cap (\mathbb{Z}^\omega \times [h(s)]) = \overline{D_s}.\]
This clearly shows that $\overline{D_s}$ is relatively open in $K_0$.

To prove the second equality, notice that $C_\sigma\subseteq C_s$ for any sequence $\sigma \supseteq s$ and thus $D_s\subseteq C_s\times\{0,1\}^\omega$. Taking closure and then intersecting with $\overline{D_s} = K_0\cap (\mathbb{Z}^\omega \times [h(s)])$ yields that $\overline{D_s} = K_0\cap (C_s \times [h(s)])$, as stated.
\end{proof}

We will use the following lemma to prove that certain subsets of $K_0$ are meager (and in fact, nowhere dense):

\begin{lemma}\label{nwdensecond}
If $X\subseteq \mathbb{Z^\omega}$ satisfies that
\[ \{s'\in \omega^{<\omega} : X \cap C_{s'} = \emptyset \} \text{ is cofinal in $\omega^{<\omega}$},\]
then $f_0^{-1}(X)$ is nowhere dense in $K_0$.
\end{lemma}

\begin{proof}
Let $U$ be an arbitrary nonempty, relatively open subset of $K_0$. Then $U$ can be written as $U= K_0 \cap V$ where $V$ is open (in $\mathbb{Z}^\omega \times \{0,1\}^\omega$).

$K_0 = \overline{D_\emptyset}$ and $K_0$ intersects the open set $V$, therefore $D_\emptyset$ also intersects $V$. This implies that there is a sequence $s\in\omega^{<\omega}$ and a point $x\in C_s$ such that $(x, h_0(s)) \in V$. 

As $V$ is an open set in the product space $\mathbb{Z}^\omega \times \{0,1\}^\omega$, there are $n, k\in\omega$ such that $[x\restrict n]\times[h_0(s)\restrict k]\subseteq V$.

As $C_s\cap [x\restrict n]$ is a nonempty relatively open subset of $C_s$, we may apply \autoref{csprop} to get an infinite set $L\subseteq \omega$ of indices such that $C_{s\concat \ell} \subseteq C_s \cap [x\restrict n]$ for each $\ell\in L$. Property (3) of \autoref{hsprop} implies we may choose an index $\ell\in L$ which also satisfies that $[h(s\concat \ell)] \subseteq [h_0(s)\restrict k]$.

Using the condition of the lemma we can find an $s'\in\omega^{<\omega}$ such that $s'\supseteq s\concat \ell$ and $X\cap C_{s'}=\emptyset$. \autoref{oldsclopen} states that $\overline{D_{s'}} = K_0\cap (C_{s'}\times [h(s')])$ is a relatively clopen subset of $K_0$. This implies that $f_0(\overline{D_{s'}}) \subseteq C_{s'}$, but then applying $X\cap C_{s'}= \emptyset$ yields $f_0^{-1}(X)\cap \overline{D_{s'}}=\emptyset$. Also notice that
\begin{align*}
\overline{D_{s'}} &= K_0 \cap (C_{s'}\times [h(s')]) \subseteq K_0\cap (C_{s\concat \ell}\times [h(s\concat\ell)]) \subseteq\\ 
&\subseteq K_0\cap ([x\restrict n] \times [h_0(s)\restrict k]) \subseteq K_0 \cap V = U.
\end{align*}

This means that in an arbitrary nonempty, relatively open subset $U$ of $K_0$ we found a (clearly nonempty) subset $\overline{D_{s'}}$ that is relatively open in $K_0$ and disjoint from $f_0^{-1}(X)$. This implies that $f_0^{-1}(X)$ is nowhere dense in $K_0$.
\end{proof}

\section{Construction of the example and proof of the main result}

Now we are ready to prove our main result:

\begin{restatemain}\label{re:mainthm}
In the abelian Polish group $\mathbb{Z}^\omega$, there exists a $G_\delta$ set $R$ that is Haar meager but not strongly Haar meager.
\end{restatemain}

\begin{proof}
Fix a map $t : \mathcal{K}(\mathbb{Z}^\omega)\to \mathbb{Z}^\omega$ which satisfies the conditions of \autoref{translateapart}. Recall that
\[H=\{(K,x) \in \mathcal{K}(\mathbb{Z}^\omega) \times \mathbb{Z}^\omega: x \in K\}\]
is a closed set and according to \autoref{translateapart} the map
\[T: H \to \mathbb{Z}^\omega,\quad T(K,x)=x+t(K)\] 
is a homeomorphism between $H$ and the closed set $F=T(H)\subseteq \mathbb{Z}^\omega$, and this set $F$ satisfies that
\[(K+t(K)+\{-1,0,1\}^\omega)\cap F = K+t(K)\]
for each $K\in\mathcal{K}(\mathbb{Z}^\omega)$. As
\[\{-1, 0, 1\}^\omega = \{0,1\}^\omega-\{0,1\}^\omega = C_\emptyset -C_\emptyset,\]
we can reformulate this fact:
\begin{fact}\label{cncnnote}
The set $F$ satisfies that
\[(K+t(K)+C_\emptyset-C_\emptyset)\cap F = K+t(K)\qquad\text{for each $K\in\mathcal{K}(\mathbb{Z}^\omega)$}.\]
\end{fact}

We will use the following definition:
\begin{definition}\label{tsdef}
If $s\in\omega^{<\omega}$, then let $T_s = \{t \in \mathbb{Z}^\omega : t + C_s \subseteq F\}$ denote the set of translations which move $C_s$ into $F$.
\end{definition}

\begin{claim}\label{tsprop}
$T_s$ is closed for each $s\in\omega^{<\omega}$.
\end{claim}

\begin{proof}
We prove that for an arbitrary $s\in\omega^{<\omega}$, the set $T_s$ contains all of its limit points. Assume that $t_i\in T_s$ for each $i\in\omega$ and $t^*=\lim_{i\to\infty} t_i$ exists. We know that if $c\in C_s$ and $i\in\omega$, then $t_i + c\in F$. As $F$ is closed, this implies that if $c\in C_s$, then $\lim_{i\to\infty} (t_i +c) = t^* + c\in F$. This shows that $t^*\in T_s$, concluding the proof.
\end{proof}

Using these sets of translations, we can define the set $R$:

\begin{definition}\label{rdef}
Let\[R=F \setminus \bigcup_{s\in\omega^{<\omega}}\bigcup_{\ell\in\omega} (T_s+C_{s\concat\ell}).\]
\end{definition}
It is clear from this definition that $R$ is $G_\delta$ and $R\subseteq F$.

The rest of the proof consists of two parts: We will first prove \autoref{rhm}, which will imply that $R$ is Haar meager, then we will prove \autoref{rnotshm}, which will imply that $R$ is not strongly Haar meager.

\begin{claim}\label{rhm}
For the compact metric space $K_0$ and function $f_0$ defined in \autoref{kzfzdef}, if $g\in\mathbb{Z}^\omega$, then $f_0^{-1}(R+g)$ is a nowhere dense subset of $K_0$. (Using the terminology of \cite{BGJS}, this states that $R$ is Haar nowhere dense.)
\end{claim}

\begin{proof}
Fix an arbitrary $g\in\mathbb{Z}^\omega$. According to \autoref{nwdensecond}, it is sufficient to prove that
\[ \{s'\in \omega^{<\omega} : (R+g) \cap C_{s'} = \emptyset \} \text{ is cofinal in $\omega^{<\omega}$}.\]
We fix an arbitrary $s\in\omega^{<\omega}$ and prove that there exists a $s'\in\omega^{<\omega}$ such that $s'\supseteq s$ and $(R+g)\cap C_{s'}=\emptyset$. Clearly $(R+g)\cap C_{s'}=\emptyset$ if and only if $R \cap (C_{s'}-g) = \emptyset$.

We distinguish two cases:

\textbf{Case 1:} $C_s-g\subseteq F$.\\In this case \autoref{tsdef} implies that $-g\in T_s$. Pick an arbitrary $\ell\in\omega$ (for example, let $\ell=0$) and let $s'=s\concat \ell$. Then 
\[C_{s'}-g \subseteq T_s +C_{s'}=T_s+C_{s\concat\ell} \subseteq\bigcup_{\sigma \in\omega^{<\omega}}\bigcup_{\ell\in\omega} (T_\sigma +C_{\sigma \concat\ell})\] and therefore \autoref{rdef} implies that $(C_{s'}- g)\cap R=\emptyset$.

\textbf{Case 2:} $C_s-g\nsubseteq F$.\\In this case the set $U=C_s\setminus (F+g)$ is nonempty and relatively open in $C_s$ (because $F$ is a closed subset of $\mathbb{Z}^\omega$). Applying \autoref{csprop} we can select an index $\ell\in\omega$ such that $C_{s\concat\ell}\subseteq U$. This means that $s'=s\concat\ell$ is a good choice:
\[C_{s'}-g \subseteq U-g = (C_s -g) \setminus F \subseteq (C_s-g) \setminus R\]
using the fact that $R\subseteq F$.
\end{proof}

\begin{claim}\label{rnotshm}
If $K\subseteq \mathbb{Z}^\omega$ is a nonempty compact set, then there is an element $g\in\mathbb{Z}^\omega$ such that $(R+g)\cap K$ is a comeager subset of $K$.
\end{claim}

\begin{proof}
Fix a nonempty compact set $K\in\mathcal{K}(\mathbb{Z}^\omega)$. We will prove that $g=-t(K)$ satisfies that $(R-t(K))\cap K$ is a comeager subset of $K$. (Recall that we fixed a map $t$ that satisfies the conditions of \autoref{tsdef}.) As $x\mapsto x+t(K)$ is a homeomorphism from $K$ to $K+t(K)$, it is enough to prove that $R\cap (K+t(K))$ is a comeager subset of $K+t(K)$.

Recall that according to \autoref{rdef},
\[R=F \setminus \bigcup_{s\in\omega^{<\omega}}\bigcup_{\ell\in\omega} (T_s+C_{s\concat\ell}).\]
It is clear that $K+t(K)\subseteq F$, therefore
\begin{align*}
R\cap (K+t(K)) &= (K+t(K)) \setminus \bigcup_{s\in\omega^{<\omega}}\bigcup_{\ell\in\omega} (T_s+C_{s\concat\ell})=\\
&=\bigcap_{s\in\omega^{<\omega}}\bigcap_{\ell\in\omega}((K+t(K)) \setminus  (T_s+C_{s\concat\ell})).
\end{align*}
To prove that this countable intersection is comeager in $K+t(K)$, it is enough to prove that if we fix $s\in\omega^{<\omega}$ and $\ell\in\omega$, then the set $(K+t(K))\setminus (T_s+C_{s\concat\ell})$ is comeager in $K+t(K)$. This set is clearly relatively open, because $C_{s\concat\ell}$ is compact and \autoref{tsprop} states that $T_s$ is closed. Therefore it is enough to prove that this set is dense in $K+t(K)$. If we fix an arbitrary nonempty relatively open subset $U$ of $K+t(K)$, then we need to check that
\[U\cap ((K+t(K))\setminus (T_s+C_{s\concat\ell})) = U\setminus (T_s+C_{s\concat\ell}) \neq \emptyset.\]

Fix an arbitrary $u\in U$. As $U$ is relatively open, we may find an index $n\in\omega$ such that $[u\restrict n] \cap (K+t(K))\subseteq U$. We will use recursion to define a sequence $x\in K+t(K) \subseteq \mathbb{Z}^\omega$. First let 
\begin{multistmt}
\item $x\restrict n = u\restrict n$.
\end{multistmt}
After this we will select the elements of the sequence $x$ one by one. Assume that we already defined $x\restrict j$ for an index $j\ge n$ and
\begin{multistmt}[resume]
\item if $j\equiv 0\modulo{2^{\lvert s\rvert+1}}$, then let
\[x_{j} = \min \{y_{j} : y \in (K+t(K)) \cap [x\restrict j]\},\]
\item if $j\equiv 2^{\lvert s\rvert}\modulo{2^{\lvert s\rvert+1}}$, then let
\[x_{j} = \max \{y_{j} : y \in (K+t(K)) \cap [x\restrict j]\},\]
\item otherwise, choose an arbitrary element $x_j$ which satisfies that
\[x_{j} \in \{y_{j} : y \in (K+t(K)) \cap [x\restrict j]\}.\]
\end{multistmt}
It is easy to check that $(K+t(K))\cap [x\restrict j]$ remains nonempty during this procedure and therefore $x\in K+t(K)$. Note that in conditions (2) and (3) the minimum and maximum are well-defined because $(K+t(K)) \cap [x\restrict j]$ is compact.

Property (1) implies that $x\in U$, we wish to prove that $x\notin T_s+C_{s\concat\ell}$. Assume for the contrary that $x = t^* + c^*$ for some elements $t^*\in T_s$ and $c^*\in C_{s\concat\ell}$.

Recall that
\[T_s = \{t \in \mathbb{Z}^\omega : t+ C_s\subseteq F\}\qquad\text{(\autoref{tsdef})}\]
and
\[C_s = \{ b_s\concat x : x\in\{0,1\}^\omega\text{ and $x$ is $\lvert s\rvert$-segmented}\}\qquad\text{(\autoref{csdef})}.\]
$t^*\in T_s$ means that $t^*+C_s=x-c^*+C_s\subseteq F$. As $x\in K+t(K)$, $c^*\in C_{s\concat\ell}\subseteq C_\emptyset$ and $C_s\subseteq C_\emptyset$, we know that $x-c^*+C_s\subseteq K+t(K)+C_\emptyset-C_\emptyset$. According to \autoref{cncnnote},
\[(K+t(K) +C_\emptyset-C_\emptyset) \cap F = K+t(K),\]
therefore $t^*+C_s = x-c^*+C_s\subseteq K+t(K)$.

For each index $j\in\omega$, consider the sequences
\[r(j,0) = (c^*\restrict j) \concat (0,0,\ldots)\quad\text{and}\quad r(j,1) = (c^*\restrict j) \concat (1,1,\ldots)\]
consisting of the first $j$ elements of $c^*$, followed by zeroes and ones respectively. If $j\ge \lvert b_s\rvert$ and $j$ is divisible by $2^{\lvert s\rvert}$ then it is straightforward to check that $r(j,0), r(j,1)\in C_s$ (using \autoref{csdef} and the fact that $c^*\in C_{s\concat\ell}\subseteq C_s$).

Fix a $j\in \omega$ such that $j\equiv 0\modulo{2^{\lvert s\rvert+1}}$ and $j> \max\{n, \lvert b_{s\concat\ell}\rvert\}$. According to property (2) of $x$, 
\[x_{j} = \min \{y_{j} : y \in (K+t(K)) \cap [x\restrict j]\}.\]
Notice that $y=x-c^*+r(j,0)$ satisfies that $y \in x-c^*+C_s\subseteq K+t(K)$ (because $r(j,0)\in C_s$) and $y\in [x\restrict j]$ (because $c^*\restrict j = r(j,0)\restrict j$). This implies that
\[x_{j}\le y_{j} = (x-c^*+r(j,0))_{j}=x_{j}-c^*_{j}+0\quad\Rightarrow\quad c^*_{j} \le 0.\]

Now apply an analogous argument for the index $j'=j+2^{\lvert s\rvert}$: According to property (3) of $x$, 
\[x_{j'} = \max \{y_{j'} : y \in (K+t(K)) \cap [x\restrict j']\}.\]
Notice that $y=x-c^*+r(j',1)$ satisfies that $y \in x-c^*+C_s\subseteq K+t(K)$ (because $r(j',1)\in C_s$) and $y\in [x\restrict j']$ (because $c^*\restrict j' = r(j',1)\restrict j'$). This implies that
\[x_{j'}\ge y_{j'} = (x-c^*+r(j',1))_{j'}=x_{j'}-c^*_{j'}+1\quad\Rightarrow\quad c^*_{j'} \ge 1.\]

As $c^* = b_{s\concat\ell}\concat z$ for some $(\lvert s\rvert+1)$-segmented sequence $z\in\{0,1\}^\omega$ and the length of $b_{s\concat\ell}$ is divisible by $2^{\lvert s\rvert+1}$, we know that $c^*_{j}=c^*_{j+2^{\lvert s\rvert}}=c^*_{j'}$. This contradicts that $c^*_j\le 0$ and $c^*_{j'}\ge 1$, proving that our indirect assumption was incorrect.

We proved that $x\in U$ and $x\notin (T_s+C_{s\concat\ell})$, and this implies that $(K+t(K))\setminus (T_s+C_{s\concat\ell})$ is indeed a dense subset of $K+t(K)$.
\end{proof}
This concludes the proof of \autorefre{mainthm}.
\end{proof}

The following theorem shows that our $G_\delta$ counterexample is as simple as possible:

\begin{theorem}\label{mincompl}
If $G$ is an abelian Polish group and $A\subseteq G$ is an $F_\sigma$ Haar meager subset, then $A$ is strongly Haar meager.
\end{theorem}

\begin{proof}
Assume that the continuous map $f : K \to G$ witnesses that $A$ is Haar meager (where $K$ is a nonempty compact metric space). We show that the compact set $C=f(K)\subseteq G$ witnesses that $A$ is strongly Haar meager. Assume that the set $B$ is a translate of $A$ (that is, $B= A+g$ for some $g\in G$). It is clearly enough to prove that $B \cap C$ is meager in $C$.

We will use the following facts which are all well-known and easy to prove:
\begin{multistmt}
\item an $F_\sigma$ set is meager if and only if it has empty interior,
\item the preimage of an $F_\sigma$ set under a continuous function is also $F_\sigma$,
\item if the preimage of a set $X$ under a continuous function has empty interior, then the set $X$ itself has empty interior relative to the image of the function.
\end{multistmt}

As $B$ is an $F_\sigma$ set, (2) yields that $f^{-1}(B)$ is also $F_\sigma$. As $f$ was a witness function, $f^{-1}(B)$ is meager, but then (1) yields that $f^{-1}(B)$ has empty interior. But $f^{-1}(B)=f^{-1}(B\cap C)$ because $C$ is the image of $f$, therefore (3) yields that $B\cap C$ has empty interior relative to $C$. Using (1) again ($B\cap C$ is an $F_\sigma$ subset of $C$), this yields that $B\cap C$ is indeed a meager subset of $C$.
\end{proof}

\begin{remark}
This proof also works when the Polish group $G$ is not necessarily abelian. For the definition of Haar meager and strongly Haar meager sets in this more general setting, see e.g. the survey paper \cite{EN}.
\end{remark}

\section{Open questions}

As we answered \autoref{que:HMeqSHM} negatively, constructing a Haar meager but not strongly Haar meager set in $\mathbb{Z}^\omega$, only the following open-ended question remains of \autoref{que:HMeqSHM}:

\begin{question}
What can we say about the (abelian) Polish groups where every Haar meager set is strongly Haar meager?
\end{question}

Notice that both \autoref{bgjsexample} and \autoref{mainthm} studied groups that can be written as countable products of countable discrete groups. In fact, our ideas allow us to describe the situation in this frequently studied, simple class of Polish groups:

\begin{claim}
Consider a group $G = \prod_{i\in\omega} G_i$ where each $(G_i, +)$ is a countable abelian group endowed with the discrete topology. Then the following are equivalent:
\begin{multistmt}
\item every Haar meager subset of $G$ is strongly Haar meager,
\item every compact subset of $G$ is contained in a locally compact subgroup,
\item for all but finitely many $i\in\omega$ the group $G_i$ is a torsion group.
\end{multistmt}
\end{claim}

We do not include a proof of this claim, as our proof involves proving generalizations of \autoref{translateapart} and \autoref{rnotshm} that are more complicated and harder to understand, but do not require additional interesting ideas.

\begin{remark}
In the previous claim the implication $(2)\Rightarrow (1)$ remains true for any abelian Polish group $G$. This can be proved by slightly modifying the proof of \cite[Theorem 5.13]{BGJS}.
\end{remark}

In addition to \autoref{que:HMeqSHM} the paper \cite{Da} contains another problem about strongly Haar meager sets, \cite[Problem 3]{Da}:

\begin{question}[Darji]
Does the collection of strongly Haar meager sets form a $\sigma$-ideal?
\end{question}

In fact, even the following variant of this question seems to be open:

\begin{question}
Does the collection of strongly Haar meager sets form an ideal? (Or equivalently, is the union of two strongly Haar meager sets strongly Haar meager?)
\end{question}

As the Haar meager sets form a $\sigma$-ideal, the interesting case for these questions is studying groups that contain Haar meager, but not strongly Haar meager sets.

\end{document}